\documentclass{article}%
\usepackage{amsfonts}
\usepackage{amsthm}
\usepackage{amsmath}
\usepackage{amssymb}
\usepackage{graphicx}%
\usepackage{epsfig}
\usepackage{color}
\usepackage{bbm}
\providecommand{\U}[1]{\protect\rule{.1in}{.1in}}

\def\subneq{\mathop{\raise 0.7ex \hbox{$\subset$}}\!\!\!\!\!\!{\raise -0.6ex\hbox{$\neq$}}\,}

\def\div{\mathop{\rm div}\nolimits}

\def\u{\underline}

\def\QQ{{\ \rlap {\raise 0.4ex \hbox{$\scriptscriptstyle |$}}\hskip -0.2em Q}}

\def\1{{1\hskip-0.25em{\rm l}}}

\def\CC{{\ \rlap{\raise 0.4ex \hbox{$\scriptscriptstyle |$}}\hskip -0.2em C}}

\def\sobre#1#2{\lower 1ex \hbox{ $#1 \atop #2 $ } }
\def\bajo#1#2{\raise 1ex \hbox{ $#1 \atop #2 $ } }

\def\div{{\rm div}}
\def\ep{\varepsilon}

\def\p{\partial}

\def\O{\Omega}
\def\Oe{\Omega^\ep}

\def\u2{{u^\ep \over \ep^2 }}
\def\u3{{\displaystyle {\bar u}^\ep \over \ep^2 }}

\def\div{\mathop{\rm div}\nolimits}

\def\u{\underline}

\def\QQ{{\ \rlap {\raise 0.4ex \hbox{$\scriptscriptstyle |$}}\hskip -0.2em Q}}

\def\1{{1\hskip-0.25em{\rm l}}}

\def\CC{{\ \rlap{\raise 0.4ex \hbox{$\scriptscriptstyle |$}}\hskip -0.2em C}}

\def\sobre#1#2{\lower 1ex \hbox{ $#1 \atop #2 $ } }
\def\bajo#1#2{\raise 1ex \hbox{ $#1 \atop #2 $ } }

\def\div{{\rm div}}
  
\def\ve{\mathbf{v}^\ep}
\def\ep{\varepsilon}

\def\p{\partial}

\def\O{\Omega}
\def\Oe{\Omega^\ep}

\def\u2{{u^\ep \over \ep^2 }}
\def\u3{{\displaystyle {\bar u}^\ep \over \ep^2 }}
\def\p{\partial}
\newtheorem{theorem}{Theorem}

\newtheorem{corollary}[theorem]{Corollary}

\newtheorem{lemma}[theorem]{Lemma}

\newtheorem{proposition}[theorem]{Proposition}
\newtheorem{remark}[theorem]{Remark}

\begin{document}

\title{A nonlinear effective slip interface law for transport phenomena between a fracture flow and a porous medium}
\author{Anna Marciniak-Czochra \thanks{AM-C  was supported by ERC Starting Grant "Biostruct" and Emmy Noether Programme of German Research Council (DFG).}
\\Institute of Applied Mathematics\\
Interdisciplinary Center of Scientific Computing and BIOQUANT,\\
University of Heidelberg \\
Im Neuenheimer Feld 294,
69120 Heidelberg  \\ GERMANY \\ ({\tt anna.marciniak@iwr.uni-heidelberg.de})
\and Andro Mikeli\'c\thanks{The research of A.M.
was partially supported by  {  the  Programme Inter Carnot Fraunhofer from BMBF (Grant 01SF0804) and ANR.}. He is grateful to the Ruprecht-Karls-Universit\"at Heidelberg and the Heidelberg Graduate School of Mathematical and Computational Methods for the Science (HGS MathComp) for giving him good working conditions through the W. Romberg Guest Professorship 2011-2013.}
 \\  Universit\'e de Lyon, CNRS UMR 5208,\\
  Universit\'e Lyon 1, Institut Camille Jordan, \\   43, blvd. du 11 novembre 1918,
 69622 Villeurbanne Cedex,\\ FRANCE \\(\texttt{mikelic@univ-lyon1.fr})  } \maketitle

\begin{abstract}
We present modeling of an incompressible viscous flow through a fracture
adjacent to a porous medium. We consider a fast stationary flow,
predominantly tangential to the porous medium. Slow flow in such setting can
be described by the Beavers-Joseph-Saffman slip. For fast flows, a nonlinear
filtration law in the porous medium and a non- linear interface law are
expected. In this paper we rigorously derive a quadratic effective slip
interface law which holds for a range of Reynolds numbers and fracture
widths. The porous medium flow is described by the Darcys law. The result
shows that the interface slip law can be nonlinear, independently of the
regime for the bulk flow. Since most of the interface and boundary slip laws
are obtained via upscaling of complex systems, the result indicates that
studying the inviscid limits for the Navier-Stokes equations with linear slip law at the boundary should be rethought.
\end{abstract}

\section{Introduction}

 Coupling between a fast viscous  incompressible fracture flow  and an adjacent filtration through porous medium  occurs in a wide range of industrial processes and natural phenomena. The classical approach is to model the fracture flow using the lubrication approximation and to replace it by an interface condition.
 Subsequently, it is coupled with a porous medium flow,  described for small Reynolds numbers  by the Darcy's law and by the Forchheimer's law in  the case of large Reynolds' number.

Study of the coupling between slow viscous incompressible fracture flow and a porous medium was undertaken in \cite{BMPAM:94} and \cite{BMPAM:95}. For the critical fracture width, the interface condition linked to the Reynolds' equation from lubrication was found.

To describe a contact between a porous medium and a large  fracture with the width significantly larger than the pore size, the following effective slip interface law was established in the seminal work by Beavers and Joseph \cite{BJ},
\begin{equation}\label{AUX}
    \sqrt{K} \frac{\partial v_{\tau}}{ \partial \mathbf{n}} =\alpha_{BJ}  v_\tau +
O(K),
\end{equation}
where $\alpha_{BJ}$ is a dimensionless parameter
depending on the geometrical structure of the porous medium%, $\ep$is the characteristic pore size,
   and $ K$ is the scalar permeability. $v_\tau$ is the tangential velocity and $\mathbf{n}$ is the unit normal exterior to the fluid region. Note that in the original version of the law (\ref{AUX}), $v_{\tau}$ was replaced by the difference between $v_\tau$ and the tangential Darcy velocity at the interface. In \cite{SAF}, Saffman remarked that  the tangential Darcy velocity at the interface is of order $O(K)$. Then, the slip law without the tangential Darcy velocity at the interface (\ref{AUX}) became generally accepted.

The rigorous derivation of the law by Beavers and Joseph through a homogenization limit and by constructing the interface boundary layer  was done by J\"ager and colleges in \cite{JaMi2}, \cite{JM00} and \cite{JMN01}. The pressure jump at the interface was studied analytically in  \cite{AMCAM2011} and using numerical simulations in \cite{CGMCM:13}. For the review of the results we refer to \cite{JaegMik09}, \cite{MIK00} and \cite{DQ09}.

Sahraoui and Kaviany investigated in \cite{SahKav92} a flow at the interface between a fracture and a porous medium by direct numerical simulations. The interest of this work was in the interface laws in presence of large Reynolds' numbers.   The interface slip behavior in that case turned out to be complex. It was concluded that the flow inertia effects appear independently from the bulk nonlinear filtration in the porous medium. If $\ep$ is a characteristic nondimensional pore size, then for longitudinal Reynolds' numbers of order $O(1/ \ep )$,  numerical simulations indicate that the slip law ceases to be linear. The inertia forces at the interface become significant for Reynolds' numbers of order $O( 0.1 / \ep )$. Then, the slip coefficient $\alpha_{BJ}$ increases. For the bulk porous medium flow, the nonlinear effects become visible only for Reynolds' numbers greater than $O(3/ \ep)$.    { Those observations led to a conclusion that  $\alpha_{BJ}$ depends on the Reynolds' number, \cite{K95} and \cite{IL}.  Similar conclusion   is in \cite{LiuP:2011}.}

{ However, it seems that a linear slip law, even with the slip coefficient depending on Reynolds' number,
is not enough for an accurate approximation and that a nonlinear slip law should be derived. We will justify it by constructing rigorously an accurate approximation to the velocity field and showing that it leads to a quadratic slip law.}

{ In the present paper we aim to identify a setting corresponding to a nonlinear slip law. We show that
 for a range of values of Reynolds' number and fracture width, the homogenization leads to a nonlinear interface law, even though  the bulk filtration remains of the Darcy type. To streamline the presentation, we focus on a mathematical model in a simple setting. We consider a constant driving force, present only in the fracture and{, for simplicity, impose }periodic longitudinal boundary conditions { for the velocity and for the pressure}. Such simplification allows to avoid handling the pressure field and the outer boundary layers. The general case of nonstationary flows with physical boundary conditions and forcing terms will be considered in forthcoming papers.} %

The paper is organized as follows: In section \ref{MR}, we define the problem as a stationary incompressible Navier-Stokes flow with Reynolds' number of the order $\ep^{-\gamma}$ and the fracture width of the order $\ep^{\delta}$. Assuming a relation between $\gamma$ and $\delta$, allows us to obtain an approximation which satisfies a nonlinear slip law (\ref{Av3}), while keeping a linear filtration in a porous medium. In section \ref{BJSec} we construct the approximation and prove that it provides a higher order approximation to the original problem.

 \section{Main result}\label{MR}

\subsection{Geometry}

 We consider a two dimensional periodic porous medium $\Omega_2 =
(0,1)\times (-1, 0)$ with a periodic arrangement of the pores. The formal
description goes along the following lines: \vskip0pt
First, we define the geometrical
structure inside the unit cell $Y = (0,1)^2$. Let
$Y_s$ (the solid part) be a closed strictly included subset of $\bar{Y}$, and $Y_F =
Y\backslash Y_s$ (the fluid part). Then,  we introduce a periodic
repetition of $Y_s$ all over $\mathbb{R}^2$ and set $Y^k_s = Y_s + k $, $k
\in \mathbb{Z}^2$. Obviously, the resulting set $E_s = \bigcup_{k \in
\mathbb{Z}^2} Y^k_s$ is a closed subset of $ \mathbb{R}^2$ and $E_F =  \mathbb{R}^2
\backslash E_s$ in an open set in $ \mathbb{R}^2$.  We suppose that
 $Y_s$ has a smooth boundary. 
  Consequently,  $ E_F $ is connected and
$E_s$ is not.  Finally, we notice that $\Omega_2$ is covered with a
regular mesh of size $ \varepsilon$, each cell being a cube
$Y^{\varepsilon}_i$, with $1 \leq i \leq N(\varepsilon) = \vert
\Omega_2 \vert  \varepsilon^{-2} [1+ o(1)]$. Each cube
$Y^{\varepsilon}_i$ is homeomorphic to $Y$, by linear homeomorphism
$\Pi^{\varepsilon}_i$, being composed of translation and a homothety
of ratio $1/ \varepsilon$.

We define
$\displaystyle
Y^{\varepsilon}_{S_i} = (\Pi^{\varepsilon}_i)^{-1}(Y_s)$ and
$Y^{\varepsilon}_{F_i} =
(\Pi^{\varepsilon}_i)^{-1}(Y_F).
$
For sufficiently small $\varepsilon > 0 $, we consider a set
$ \displaystyle T_{\varepsilon} = \{k \in  \mathbb{Z}^2  \vert  Y^{\varepsilon}_{S_k} \subset
\Omega_2 \} $
and define
$$
O_{\varepsilon} = \bigcup_{k \in T_{\varepsilon}}
Y^{\varepsilon}_{S_k} , \quad S^{\varepsilon} = \partial
O_{\varepsilon}, \quad \Omega^{\varepsilon}_2 = \Omega_2 \backslash
O_{\varepsilon} =\O_2 \cap \ep E_F.$$
Obviously, $\partial
\Omega^{\varepsilon}_2 = \partial \Omega_2 \cup S^{\varepsilon}$. The
domains $O_{\varepsilon}$ and $\Omega^{\varepsilon}_2 $ represent the solid and the fluid part of the porous medium $\Omega$, respectively. For simplicity, we assume $1/\varepsilon  \in \mathbb{N}$.

Let $0<\delta <1$. We set $\Sigma =(0,1) \times \{ 0\} $, $\Omega_1^{\ep , \delta} = (0,1)\times (0, \ep^\delta)$ and $\O = (0,1) \times (-1, \ep^\delta)$.
Furthermore, let $\O^\ep = \Oe_2 \cup \Sigma \cup \O_1^{\ep , \delta} $.

In such geometry, homogenization of the Stokes equation with no-slip boundary conditions on $S^\ep$ leads to Darcy law (see \cite{{All97}}, \cite{ESP}, \cite{SP80} and \cite{Ta1980}). In the presence of inertia, a nonlinear corrections to Darcy law arise, as studied in \cite{BMPAM:96}.
\subsection{Position of the problem and the nonlinear slip law}

Let $0<\gamma <3/2$ and let $F$ be a constant. In $\O^\ep$ we study the following stationary Navier-Stokes equation
\begin{gather} -
 \ep^\gamma \Delta \mathbf{v}^\ep +( \mathbf{v}^\ep \nabla ) \mathbf{v}^\ep + \nabla p^{\ep} = F \mathbf{e}^1 { \mathbbm{1}_{\{ x_2 >0 \} }} \qquad \hbox{ in } \quad \Oe
\label{1.3} \\ \div \, \ve = 0 \qquad \hbox{ in } \quad \Oe , \qquad \int_{\O_1^{\ep , \delta}} p^\ep \ dx =0,
\label{1.4} \\ \ve  =0 \quad \hbox{on } \quad \p \Oe \setminus \bigg(  \{ x_1 = 0 \} \cup \{ x_1 =1 \} \bigg)
, \qquad \{ \ve , p^\ep \} \quad \hbox{ is }
1-\hbox{periodic in } \; x_1 . \label{1.5} \end{gather}

\begin{remark}
We skip here a discussion of modeling aspects.  We only mention that $\ep^\gamma$ stands for the inverse of Reynolds' number { and that the small fracture width $\ep^\delta$ prevents creation of the Prandtl's boundary layer}.
\end{remark}
In order to simplify calculations we take a constant $F$. It corresponds to a pressure drop. Additionally,  we assume it only in the fracture $\O_1^{\ep , \delta}$. Let
\begin{equation}\label{WE}
   W^\ep = \{ \mathbf{z} \in H^1 (\Oe )^2 , \;
\mathbf{z}=0 \; \hbox{ on } \; \p \Oe \setminus  \bigg(  \{ x_1 = 0 \} \cup \{ x_1 =1 \} \bigg) \; \hbox{ and }
\; \mathbf{z} \; \hbox{ is } \; 1-\hbox{periodic in } x_1\} .
\end{equation}
The variational form of problem (\ref{1.3})-(\ref{1.5}) reads:\\

Find $\ve \in W^\ep  $, div $\ve =0$
in $\Oe$ and $p^\ep \in L^2 (\Oe )$ such that
\begin{equation}\label{1.6}
    \int_{\Oe} \ep^\gamma \nabla \ve \nabla \varphi \, dx  + \int_{\Oe} (\ve \cdot \nabla ) \ve  \varphi \, dx - \int_{\Oe} p^\ep \hbox{ div }
\varphi \, dx = \int_{\O_1^{\ep , \delta}} F  \varphi_1 \, dx, \qquad \forall \varphi
\in W^\ep .
\end{equation}

Theory of the stationary Navier-Stokes equations with homogeneous boundary conditions results in existence of the least one smooth velocity
field  $\ve \in W^\ep$ ,
 div $\ve =0$ in $\Oe$, which solves (\ref{1.6}) for every $ \varphi \in W^\ep
 , $ div $\varphi =0$ in $\Oe$. The construction of the pressure field
goes through De Rham's theorem. For more details we refer to the classical Temam's book \cite{Tem}.
\vskip10pt

 Now we make assumptions on the parameters $\delta$ and $\gamma$.
 \begin{description}
   \item[(H1)] $2\gamma < 3 \delta $,
     \item[(H2)] $0<\delta<1 $ and $0<\gamma <3/2 ,$
   \item[(H3)] $4\delta <2 \gamma + 1$.
 \end{description}
 Now, we can formulate the main result
 %\vskip25pt
  \begin{theorem} \label{T4.177}Let us suppose the hypothesis ({\bf H1})-({\bf H3}) and  let ${\cal U}^{2,\ep }$    be defined by
\begin{gather}
{\cal U}^{2,\ep } =  \ve +{\ep^{2\delta -\gamma}  \frac{F }{ 2    }} \frac{x_2}{\ep^\delta} (\frac{x_2}{\ep^\delta} -
1) \mathbf{e}^1 + \frac{F}{2} \ep^{\delta +1 -\gamma} \beta^{ bl} (\frac{x }{ \ep})  - \frac{F}{2} \ep^{\delta +1 -\gamma} C^{bl}_1 {
\frac{x_2 }{ \ep^\delta}}  \mathbf{e}^1 \notag \\ - \frac{F}{2} C^{bl}_1 \ep^{2 -\gamma} \beta^{ bl} (\frac{x }{ \ep})  + \frac{F}{2} \ep^{2 -\gamma} (C^{bl}_1)^2 {
\frac{x_2 }{ \ep^\delta}}  \mathbf{e}^1
  +(\frac{F}{2})^2 \ep^{2\delta +3 -3\gamma} \beta^{1, bl} (\frac{x }{ \ep})  {-} (\frac{F}{2})^2 \ep^{2\delta +3 -3\gamma} C^{bl}_{11} {
\frac{x_2 }{ \ep^\delta}}  \mathbf{e}^1 ,
\label{coortotal}
\end{gather}
where the boundary layer functions $\beta^{bl}$ and $\beta^{1, bl}$ are defined, respectively, by (\ref{BJ4.2})-(\ref{4.6}) and (\ref{BLRRR5.2})-(\ref{BLRRR5.5}). The constant $C^{bl}_1 <0$ is the stabilization constant for $\beta^{bl}_1$ when $y_2 \to +\infty.$ Similarly $C^{bl}_{11}$ is the stabilization constant for $\beta^{1,bl}_1$ when $y_2 \to +\infty.$\vskip1pt
 Then, the following
estimate holds
\begin{gather}
\ep \Vert \nabla {\cal U}^{2,\ep }   \Vert_{L^2 (\Oe )^4} + \Vert  {\cal U}^{2,\ep }  \Vert_{L^2 (\Oe_2 )^2} +
\ep^{1/2} \Vert  {\cal U}^{2,\ep }  \Vert_{L^2 (\Sigma  )^2} +\ep^{1-\delta} \Vert  {\cal U}^{2,\ep }  \Vert_{L^2 (\O_1^{\ep, \delta} )^2}
 \leq C\ep^{7/2 -\delta -\gamma} .
 \label{4.885}
 \end{gather}
\end{theorem}
  \begin{remark} { The rigorous result from Theorem \ref{T4.177}}, showing that ${\cal U}^{2,\ep}$ is of order $O(\ep^{3-\delta -\gamma})$ { on $\Sigma$}, allows  justifying  a nonlinear interface law. Contrary to the classical situation, when  Saffman's
modification  of the linear slip law by Beavers and Joseph  (see \cite{BJ} and \cite{SAF}) is used, the nonlinear interface laws are rarely derived in the literature.  However, they are supposed to be appropriate for fast flows.

Setting $\delta = 1- 7 \eta /12$ and $\gamma =3/2 - \eta$, where $0<\eta < 3/2$, which fulfills hypotheses ({\bf H1})-({\bf H3}), we obtain
on the interface $\Sigma$
\begin{eqnarray*}
    v_1 (\ep) |_\Sigma &=&  - \frac{F}{2} \ep^{\delta +1 -\gamma} (1- C^{bl}_1 \ep^{1 -\delta} ) \beta^{ bl} (\frac{x }{ \ep})|_\Sigma  -(\frac{F}{2})^2 \ep^{2\delta +3 -3\gamma} \beta^{1, bl}_1 (\frac{x }{ \ep}) |_\Sigma
       \\
      &=&- \frac{F}{2} \sqrt{\ep} \ep^{5\eta /12} (1- C^{bl}_1 \ep^{7\eta/12} ) \beta^{ bl}_1 (\frac{x }{ \ep})|_\Sigma  -(\frac{F}{2})^2 \sqrt{\ep} \ep^{11\eta /6 } \beta^{1, bl}_1 (\frac{x }{ \ep}) |_\Sigma
\end{eqnarray*}
and for the average over the pore face on $\Sigma$
\begin{gather}
< v_1 (\ep)  |_\Sigma > = v_1^{eff} = - \frac{F}{2} \sqrt{\ep} \ep^{5\eta /12} (1- C^{bl}_1 \ep^{7\eta/12} )  C^{bl}_1 -(\frac{F}{2})^2 \sqrt{\ep} \ep^{11\eta /6 } <\beta^{1, bl}_1 (\frac{x }{ \ep}) |_\Sigma >.
\label{Av1}
\end{gather}
Next, for the shear stress we have
\begin{gather*}
 \frac{\p v_1 (\ep)}{\p x_2} |_\Sigma =  \ep^{\delta -\gamma} \frac{F}{2}
 - \ep^{\delta -\gamma} \frac{F}{2} \frac{\p \beta^{bl}_1 }{ \p y_2} |_{\Sigma , y=x/\ep}  +\ep^{1 -\gamma} \frac{F}{2} C^{bl}_1 + \frac{F}{2} C^{bl}_1 \ep^{1 -\gamma} \frac{\p \beta^{bl}_1 }{ \p y_2} |_{\Sigma , y=x/\ep} \\
 - \ep^{2-\delta -\gamma} \frac{F}{2} (C^{bl}_1)^2
  -(\frac{F}{2})^2 \ep^{2\delta +2 -3\gamma} \frac{\p \beta^{1, bl}_1 (\frac{x }{ \ep})}{\p y_2} |_\Sigma {+} (\frac{F}{2})^2 \ep^{\delta +3 -3\gamma} C^{bl}_{11}.
\end{gather*}
 After averaging over $\Sigma$ with respect to $y_1$, we
obtain
\begin{gather}
<\frac{\p v_1 (\ep)}{\p x_2} |_\Sigma <= \frac{\p v_1^{eff}}{\p x_2} =  \frac{F}{2} \ep^{-1/2 + 5\eta /12} ( 1+ \ep^{7\eta /12} C^{bl}_1 - \ep^{7\eta /6} (C^{bl}_1)^2 ) -\notag \\
 (\frac{F}{2})^2 \ep^{-1/2 + 11\eta /6} ( <\frac{\p \beta^{1, bl}_1 (\frac{x }{ \ep})}{\p y_2} |_\Sigma > { - \ep^{7\eta /12} C^{bl}_{11}} ).
\label{Av2}
\end{gather}
Next, elimination of $F/2$ yields {
\begin{gather}
v_1^{eff} = - C^{bl}_1 \ep \frac{\p v_1^{eff}}{\p x_2} \frac{1 - C^{bl}_1 \ep^{7\eta /12}}{1+ C^{bl}_1 \ep^{7\eta /12} (1- C^{bl}_1 \ep^{7\eta /12})}  \notag \\
-  \ep^{3/2 +\eta}  < \beta^{1, bl} (\frac{x }{ \ep}) |_\Sigma > (\frac{\p v_1^{eff}}{\p x_2})^2 + \;
%\mbox{higher order powers of} \quad
O(\ep^{3/2+19\eta /12}) .
\label{Av3}
\end{gather}}
The above formula results in Saffman' version of the law by
Beavers and Joseph,  if only the first term at the { right} hand-side is taken into consideration. For small $\eta$, we obtain a significant deviation of the law by Beavers and Joseph from \cite{SAF} and \cite{BJ}. { We are not aware of any rigorous derivation of a nonlinear interface law for the unconfined fluid flow coupled to the porous media flow.}
   \end{remark}

\section{Rigorous justification of the nonlinear slip law, generalizing the law by Beavers and Joseph}\label{BJSec}

  \vskip0pt In this section we extend the
justification of the law of Beavers and Joseph from \cite{JM00} to the case of nonlinear laminar flows. In the proofs we apply the following variant of
Poincar\'e's inequality:

\begin{lemma} \label{L1} (see e.g. \cite{SP80})
Let $\varphi \in V(\Oe_2 ) = \{ \varphi \in H^1 (\Oe_2 ) \ | \varphi =0
 \; \hbox{ on } {S}^\ep \} $ and $\psi \in H^1 (\O_1^{\ep , \delta})$ such that $\displaystyle \psi |_{\{ x_2 = \ep^\delta \}} =0$.
  %  $ \varphi \; \mbox{ is periodic in } \; (x_1 , x_2) \; \hbox{ with period } \; L\ \}.$
Then, it holds
\begin{gather}
 \Vert \varphi \Vert _{L^2 (\Sigma )} \leq C \ep^{1/2} \Vert  \nabla_x
\varphi \Vert _{L^2 (\Oe_2 )^2}, \label{Poinc1} \\
 \Vert \varphi \Vert _{L^2 (\Oe_2  )} \leq C \ep \Vert  \nabla_x
\varphi \Vert _{L^2 (\Oe_2 )^2},\label{Poinc2} \\
\Vert \psi \Vert _{L^2 (\Sigma  )} \leq C \ep^{\delta /2} \Vert  \nabla_x
\psi \Vert _{L^2 (\O_1^{\ep , \delta})^2},\label{Poinc3} \\
 \Vert \psi \Vert _{L^2 (\O_1^{\ep , \delta}  )} \leq C \ep^\delta \Vert  \nabla_x
\psi \Vert _{L^2 (\O_1^{\ep , \delta})^2}.\label{Poinc4}
\end{gather}
\end{lemma}

\subsection{The impermeable interface approximation}\label{impp}
Intuitively, the main flow is in the fracture $\O_1^{\ep , \delta}$. Following the approach from \cite{JM00} we study the problem
\begin{gather}
- \ep^\gamma \triangle \mathbf{v}^0 + ( \mathbf{v}^0 \nabla ) \mathbf{v}^0 + \nabla p^0 = F \mathbf{e}^1
\qquad \hbox{ in } \O_1^{\ep , \delta} ,\label{4.37}\\
%\noalign{\vskip+4mm}
\div \  \mathbf{v}^0 = 0 \qquad \hbox{ in } \O_1^{\ep , \delta} ,\label{4.38}\\
%\noalign{\vskip+4mm}
\mathbf{v}^0 = 0  \qquad  \hbox{ on } \p   \O_1^{\ep , \delta}  \setminus \bigg(  \{ x_1 = 0 \} \cup \{ x_1 =1 \} \bigg) \quad
,\label{4.39}\\ \{ \mathbf{v}^0 , p^0 \}    \qquad  \hbox{ is } \;
1-\hbox{periodic in } \; x_1 ,  \qquad \int_{\O_1^{\ep , \delta}} p^0 \ dx =0.\label{4.40}\end{gather}
Therefore, as in \cite{JM00} and \cite{JaegMik09}, for the lowest order
approximation $\{ \mathbf{v}^0 , p^0 \}$ we impose on the interface the no-slip
condition
\begin{equation}\label{BJ3}
    \mathbf{v}^0 = 0 \qquad \mbox{on} \quad \Sigma.
\end{equation}
Such choice leads to a cut-off of the shear and it introduces an error.

 A
solution of  problem (\ref{4.37})-(\ref{4.40})  is
the classic Poiseuille flow in $\O_1^{\ep , \delta}$, satisfying the no-slip
condition at $\Sigma$. It is given by
\begin{equation}\label{4.41}
    %\begin{cases}
\mathbf{v}^0 =  -{\ep^{2\delta -\gamma}  \frac{F }{ 2    }} \frac{x_2}{\ep^\delta} (\frac{x_2}{\ep^\delta} -
1) \mathbf{e}^1
\; \hbox{  for} \quad 0\leq x_2 \leq \ep^\delta ; \qquad % \noalign{\vskip+4mm}
 p^0 = 0 \; \hbox{  for }\quad 0\leq x_1 \leq 1.% \end{cases}
\end{equation}
Concerning the normal derivative of the tangential velocity on $\Sigma$, we obtain
\begin{equation}\label{Shear0}
    \frac{\p v_1^0}{\p x_2} = { -} \ep^{\delta -\gamma} \frac{F }{ 2    } (\frac{2 x_2}{\ep^\delta} -
1) ; \qquad \frac{\p v_1^0}{\p x_2} |_{\Sigma} = \ep^{\delta -\gamma} \frac{F }{ 2    } .
\end{equation}
 We extend  $\mathbf{v}^0$  to $\O_2$ by setting
$\mathbf{v}^0 =0$ for $-1 \leq x_2 < 0$.
$p^0$ is extended by $0$ to $\O_2$. The question is in which sense this solution approximates the
solution $\{ \ve , p^\ep \}$ of the original problem
(\ref{1.3})-(\ref{1.5}).

%----------------------------------------------
A direct consequence of the weak formulation (\ref{1.6}) is that the difference $\ve - \mathbf{v}^0$ satisfies the following variational equation
\begin{gather}
 \int_{\Oe} \ep^\gamma \nabla (\ve - \mathbf{v}^0 ) \nabla \varphi \ dx
 + \int_{\Oe} \bigg( v_1^0 \frac{\p (\ve - \mathbf{v}^0 )}{\p x_1} + (v^\ep_2 - v_2^0 ) \frac{\p \mathbf{v}^0 }{\p x_2}  +
  ((\ve - \mathbf{v}^0 ) \nabla) (\ve - \mathbf{v}^0 ) \bigg) \varphi \, dx \notag \\
 -\int_{\Oe}   p^\ep \mbox{ div } \varphi   = \int_{\Sigma} \ep^\gamma  \frac{\p   v_1^0}{ \p   x_2 }   \varphi_1 \ dS ,  %\int_{\Sigma}  [{\tilde p}^0 ]  \varphi_2 \ dS + \int_{\Oe_2}(\mathbf{f} - \nabla {\tilde p}^0 ) \varphi \ dx
 \qquad  \forall \varphi \in   { W}^\ep . \label{H4.57}
\end{gather}
It leads to the following  result, which is a generalization of the result
proved in \cite{JM00}:
\begin{proposition} \label{P4.14} Let us assume that ({\bf H1})-({\bf H2}) are satisfied. Let $\{
\ve , p^\ep \}$ be a solution of (\ref{1.3})-(\ref{1.5}) and $\{ \mathbf{v}^0 , p^0 \}$ defined by (\ref{4.41}). Then, it holds for $\ep \leq \ep_0$
\begin{gather}
\sqrt{\ep} \Vert \nabla (\ve - \mathbf{v}^0 ) \Vert_{L^2 (\Oe )^4} +
%\sqrt{\ep} \Vert p^\ep  - pi^0 \Vert_{L^2 (\O_1 ) }
%\leq C\sqrt{\ep}\label{4.52}\\
  \frac{1}{\sqrt{\ep}}\Vert \ve  \Vert_{L^2 (\Oe_2 )^2} %\leq C \ep \sqrt{\ep}
 %\label{4.53}\\
 +\Vert \ve  \Vert_{L^2 (\Sigma )} + \ep^{1/2 -\delta}\Vert \ve - \mathbf{v}^0 \Vert_{L^2 (\O_1^{\ep , \delta} )^2}
\leq C \ep^{\delta -\gamma +1}
 \label{4.52}\end{gather}
\end{proposition}
\begin{proof} We test (\ref{H4.57}) with $\varphi =\ve - \mathbf{v}^0$ and obtain
\begin{equation}\label{Erreq}
     \int_{\Oe} \ep^\gamma | \nabla (\ve - \mathbf{v}^0 ) |^2 \ dx = - \int_{\Oe} (v^\ep_1 - v_0^1 )  (v^\ep_2 - v_2^0 ) \frac{\p {v}^0_1 }{\p x_2} \ dx { +} \int_{\Sigma} \ep^\gamma \frac{\p   v_1^0}{ \p   x_2 }   (v^\ep_1 - v_2^1 ) \ dS .
\end{equation}
Applying
Lemma \ref{L1} and formula (\ref{Shear0}) yield
\begin{gather*}
    | \int_{\Oe} (v^\ep_1 - v_0^1 )  (v^\ep_2 - v_2^0 ) \frac{\p {v}^0_1 }{\p x_2} \ dx |\leq C \ep^{3\delta - \gamma} \Vert \nabla (\ve - \mathbf{v}^0 )\Vert_{L^2 (\O_1^{\ep , \delta} )^4}^2 , \\
    | \int_{\Sigma} \ep^\gamma  \frac{\p   v_1^0}{ \p   x_2 }   (v^\ep_1 - v_0^1 ) \ dS  |  \leq C \ep^{\delta +1/2}
    \Vert \nabla (\ve - \mathbf{v}^0 )\Vert_{L^2 (\O_2^{\ep } )^4} .
\end{gather*}
Using hypothesis ({\bf H1}) and  above estimates lead to
$$ \int_{\Oe} \ep^\gamma | \nabla (\ve - \mathbf{v}^0 ) |^2 \ dx  \leq C \ep^{\delta +1/2}
    \Vert \nabla (\ve - \mathbf{v}^0 )\Vert_{L^2 (\O_2^{\ep } )^4} .$$
    We apply once more Lemma  \ref{L1} and (\ref{4.52}) follows.
\end{proof}
 This provides the uniform a priori
estimates for $\{ \ve , p^\ep \} $. Moreover, we have found that the viscous flow in $\O_1^{\ep , \delta}$ corresponding to an impermeable wall  is an $O(\ep^{2\delta -\gamma +1/2} )$ $L^2$-approximation
for $\ve$.  The slip law, generalizing Beavers and Joseph's law, should correspond to the next
order velocity correction. Since the Darcy velocity is of order O$(\ep^{\delta -\gamma +3/2} )$, we may
justify Saffman's observation that the { bulk} filtration effects are negligible  at this stage. \vskip2pt

\subsection{Justification of the nonlinear slip law} \label{just}

At the interface $\Sigma$ the approximation from Subsection \ref{impp} leads to
 the shear
stress jump equal to $\displaystyle \ep^\gamma \frac{\partial
v_1^{0}}{\partial x_2} |_{\Sigma} = \frac{F}{2} \ep^\delta$. The shear stress jump requires construction of the corresponding boundary layer.

The natural stretching
variable is given by the geometry and reads $\displaystyle
y=\frac{x}{\ep}$. The correction $\{ \mathbf{w} , p_w \}$ is given by
\begin{gather} -\ep^{\gamma-2} \triangle _y \mathbf{w}  + \ep^{-1} (\mathbf{w}\nabla_y) \mathbf{w}  + \ep^{-1} \nabla_y p_w =0\qquad \hbox{ in
}  \quad  \Omega_1^{\ep , \delta} / \ep  \cup \Omega_2^\ep / \ep ,\label{BJ5.2}\\
\div_y \mathbf{w} =0\qquad \hbox{ in }  \quad \Omega_1 / \ep \cup \Sigma / \ep \cup
\Omega_2^\ep / \ep , \label{BJ5.3} \\
\bigl[ \mathbf{w} \bigr] (\cdot , 0)= 0; \quad \bigl[ p_{w} \bigr] (\cdot , 0)= 0
 \quad \mbox{ and } \quad \bigr[ -\ep^{\gamma-1} \frac{\partial w_1}{\partial y_2}
\bigl] (\cdot , 0) =  \ep^\gamma \frac{\partial v_1^{0}}{\partial x_2}
|_{\Sigma} = \frac{F}{2} \ep^\delta \quad \hbox{ on }  \quad  \Sigma / \ep ,\label{BJ5.4)}\\
\nabla_y \mathbf{w} \in  L^2 (\Omega^\ep /\ep )^4  \quad \mbox{ and} \qquad \{ \mathbf{w} , p_w \} \,
\hbox{ is } 1/\ep - \hbox{periodic in } y_1 .\label{BJ5.5}
\end{gather}
It is natural to rescale $\mathbf{w}$ and $p_w$ by setting
$$ \mathbf{w}=-\ep^{\delta +1-\gamma} \frac{F}{2} \beta (y) \quad \mbox{ and } \quad p_w = -\ep^{\delta} \pi(y) \frac{F}{2}.$$
Using periodicity of the geometry and independence of $\displaystyle
\frac{\partial v_1^{0}}{\partial x_2} |_{\Sigma}$ of $y$, we obtain
\begin{gather} - \triangle _y \beta   +  \nabla_y \pi =\frac{F}{2} \ep^{\delta -2\gamma +2} (\beta\nabla_y ) \beta \qquad \hbox{ in
}  \quad  \Omega_1^{\ep , \delta} / \ep  \cup \Omega_2^\ep / \ep ,\label{BJR5.2}\\
\div_y \beta =0\qquad \hbox{ in }  \quad \Omega_1 / \ep \cup \Sigma / \ep \cup
\Omega_2^\ep / \ep , \label{BJR5.3} \\
\bigl[ \beta \bigr] (\cdot , 0)= 0; \quad \bigl[ \pi \bigr] (\cdot , 0)= 0
 \quad \mbox{ and } \quad \bigr[  \frac{\partial \beta_1}{\partial y_2}
\bigl] (\cdot , 0) =1 \quad \hbox{ on }  \quad  \Sigma / \ep ,\label{BJR5.4)}\\
\nabla_y \beta \in  L^2 (\Omega^\ep /\ep)^4  \quad \mbox{ and} \qquad \{ \beta , \pi \} \,
\hbox{ is } 1/\ep - \hbox{periodic in } y_1 .\label{BJR5.5}
\end{gather}
We do not use directly the nonlinear boundary layer problem (\ref{BJR5.2})-(\ref{BJR5.5}). Since by ({\bf H2}) we have $\delta -2\gamma +2 ${ $>0$}, we approximate $ \{ \beta , \pi \}$ with $ \displaystyle \{ \beta^0 +\frac{F}{2} \ep^{\delta -2\gamma +2} \beta^1  , \pi^0 + \frac{F}{2} \ep^{\delta -2\gamma +2} \pi^1 \}$, where the new functions are given through the following problems
\begin{gather} - \triangle _y \beta^0   +  \nabla_y \pi^0 =0 \qquad \hbox{ in
}  \quad  \Omega_1^{\ep , \delta} / \ep  \cup \Omega_2^\ep / \ep ,\label{BJRR5.2}\\
\div_y \beta^0 =0\qquad \hbox{ in }  \quad \Omega_1 / \ep \cup \Sigma / \ep \cup
\Omega_2^\ep / \ep , \label{BJRR5.3} \\
\bigl[ \beta^0 \bigr] (\cdot , 0)= 0; \quad \bigl[ \pi^0 \bigr] (\cdot , 0)= 0
 \quad \mbox{ and } \quad \bigr[  \frac{\partial \beta_1^0}{\partial y_2}
\bigl] (\cdot , 0) =1 \quad \hbox{ on }  \quad  \Sigma / \ep ,\label{BJRR5.4)}\\
\nabla_y \beta^0 \in  L^2 (\Omega^\ep /\ep)^4  \quad \mbox{ and} \qquad \{ \beta^0 , \pi^0 \} \,
\hbox{ is } 1/\ep - \hbox{periodic in } y_1 \label{BJRR5.5}
\end{gather}
and
\begin{gather} - \triangle _y \beta^1   +  \nabla_y \pi^1 = (\beta^0\nabla_y ) \beta^0 \qquad \hbox{ in
}  \quad  \Omega_1^{\ep , \delta} / \ep \cup \Sigma/ \ep \cup \Omega_2^\ep / \ep ,\label{BJRRR5.2}\\
\div_y \beta^1 =0\qquad \hbox{ in }  \quad \Omega_1 / \ep \cup \Sigma / \ep \cup
\Omega_2^\ep / \ep , \label{BJRRR5.3} \\
\nabla_y \beta^1 \in  L^2 (\Omega^\ep /\ep )^4  \quad \mbox{ and} \qquad \{ \beta^1 , \pi^1 \} \,
\hbox{ is } 1/\ep - \hbox{periodic in } y_1 .\label{BJRRR5.5}
\end{gather}
{ Because of the $1$-periodicity of the geometry with respect to $y_1$,}
problem (\ref{BJRR5.2})-(\ref{BJRR5.5}) is handled using Navier's boundary layer introduced in  \cite{JaMi2}.

{
It reads as follows:}
We introduce the interface $S=(0,1)\times \{ 0\} $, the semi-infinite slab $Z^+ = (0,1) \times (0, +\infty )$
and the semi-infinite porous slab $Z^- =
\displaystyle\cup_{k=1}^\infty ( Y_F -\{ 0,k \} )$. The flow region
is then $Z_{BL} = Z^+ \cup S \cup Z^- $. \vskip3pt { Then the following problem is  considered:}  Find $\{ \beta^{bl} , \omega ^{bl} \} $
with square-integrable gradients satisfying
\begin{gather} -\triangle _y \beta^{bl} +\nabla_y \omega ^{bl} =0\qquad \hbox{ in
} Z^+ \cup Z^- \label{BJ4.2}\\ \div_y  \beta^{bl} =0\qquad \hbox{ in }
Z^+ \cup Z^- \label{4.3} \\ \bigl[ \beta^{bl} \bigr]_S (\cdot , 0)= 0
 %\quad \hbox{ on } S \label{4.4}\\ %)\cr
 \quad \mbox{ and } \quad \bigr[ \{ \nabla_y \beta^{bl}
-\omega^{bl} I \} \mathbf{e}^2 \bigl]_S (\cdot , 0) = \mathbf{e}^1 \ \hbox{ on }
S\label{4.5)} \\ \beta^{bl} =0 \quad \hbox{ on }
\displaystyle\cup_{k=1}^{\infty} ( \p Y_s -\{ 0,k \} ), \qquad \{
\beta^{bl} , \omega^{bl} \} \, \hbox{ is } 1- \hbox{periodic in }
y_1 \label{4.6} \end{gather} By Lax-Milgram's lemma, there is  a
unique $\beta^{bl} \in L^2_{loc} (Z_{BL} )^2, \; \nabla _y \beta^{bl} \in L^2
(Z_{BL})^4$  satisfying (\ref{BJ4.2})-(\ref{4.6}) and
 $\omega^{bl} \in L^2_{loc}
(Z^+ \cup Z^- )$, unique up to a constant and satisfying
(\ref{BJ4.2}).  \vskip0pt After \cite{JaMi2}, \cite{JM00} and \cite{JMN01}, we know
 that system (\ref{BJ4.2})-(\ref{4.6}) describes a
boundary layer, i.e. that $\beta^{bl} $ and $ \omega^{bl} $
stabilize exponentially towards  constants, when $\vert y_2\vert \to
\infty$. \vskip0pt Since we are studying an incompressible flow, it
is useful to recall properties of the conserved averages.
%\begin{lemma} \label{L4.1} (\cite{JaMi2}). Any solution $\{ \beta^{bl} ,
%\omega^{bl} \}$ satisfies
%\begin{gather}
%\int^1_0 \beta^{bl}_2 (y_1 , b) \ d y_1 =0 , \quad \forall b\in
%\RR %\label{4.8}\\
%\; \mbox{ and } \; \int^1_0 \omega^{bl}  (y_1 , b_1 ) \ d y_1 =
%\int^1_0  \omega^{bl} (y_1 , b_2 ) \ d y_1 , \; \forall b_1 >
%b_2 \geq 0 \label{4.8}\\ \int^1_0 \beta^{bl}_1  (y_1 , b_1) \ d y_1
%=\int^1_0 \beta^{bl}_1   (y_1 , b_2) \ d y_1
%= - \int_{Z_{BL}} \vert \nabla \beta^{bl} (y) \vert^2 \ dy ,  \qquad \forall b_1
%> b_2 \geq 0.
%\label{4.11}\end{gather}
%\end{lemma}
\begin{proposition} \label{P4.3} (\cite{JaMi2}). Let
\begin{equation}\label{4.15}
    C^{ bl}_1  = \int_0^1 \beta^{ bl}_1 (y_1, 0) dy_1 {=-\int_{Z_{BL}} \vert \nabla \beta^{bl} (y) \vert^2 \ dy }.
\end{equation}
\begin{equation}\label{4.16}
  \hbox{Then  for every } \; y_2 \geq 0\; \hbox{ and } \; y_1 \in (0,1), \;  \vert \beta^{bl} (y_1 , y_2 ) - ( C^{bl}_1 , 0) \vert \leq C e^{-\delta y_2}, \quad \mbox{for all } \quad \delta < 2\pi .
\end{equation}
\end{proposition}
\begin{corollary} \label{C4.4} (\cite{JaMi2}). Let
\begin{equation}\label{4.19}
    C^{bl}_{\omega}  =\int_0^1  \omega^{bl} (y_1 , 0)\, dy_1 .
\end{equation}
\begin{equation}\label{4.20}
 \hbox{Then for every} \quad  y_2 \geq 0\;  \hbox{ and } \; y_1 \in (0,1), \quad \hbox{ we have } \quad   \mid \omega^{bl} (y_1 , y_2 ) - C^{bl}_\omega \mid \leq   e^{-2\pi y_2} .
\end{equation}
\end{corollary}
\begin{proposition} \label{P4.7} (\cite{JaMi2}). Let   $\beta^{bl} $ and $\omega^{bl} $
be defined by (\ref{BJ4.2})-(\ref{4.6}). Then  there exist  positive constants
$C$ and $\gamma_0$,  such that
\begin{equation}\label{4.32}
| \nabla \beta^{bl} (y_1 , y_2 ) | +
 | \nabla \omega^{bl} (y_1 , y_2 ) |  \leq C e^{-\gamma_0 | y_2|} , \qquad \hbox{for every } \quad (y_1 , y_2 ) \in Z^- .
\end{equation}\end{proposition}
  $\displaystyle \beta^{ bl } ( \frac{x}{\ep} )$ is extended by
zero to  $\O_2 \setminus \Oe $. Let $H$ be Heaviside's function. Then
for every $q\geq 1$ we have
\begin{equation}\label{4.60}
   % \begin{cases}
\Vert \beta^{ bl ,\ep} - \varepsilon (C^{bl}_1 , 0) H(x_2) \Vert
_{L^q(\O_2 \cup \O_1^{\ep , \delta})^2}
+ \Vert \omega^{bl,\ep} -C^{bl}_\omega H(x_2) \Vert _{L^q(\Oe)} +
\ep \Vert \nabla \beta^{ bl , \ep } \Vert _{L^q(\O_2 \cup \O_1^{\ep , \delta}
)^4} =C \ep^{1/q} .%, \quad \forall q\geq 1.%&\hfill \cr \end{cases}
\end{equation}
Hence, our correction is not concentrated around the interface and
there are some { nonzero} stabilization constants. We will see that these
constants are closely linked with our effective interface law.
\vskip1pt
  As in \cite{JaMi2} stabilization of $\beta^{0 , \ep }$
towards a nonzero constant velocity $  C^{bl}_1 \mathbf{e}^1 $,
at the upper boundary,  generates a counterflow. It is given by the
two dimensional Couette flow $\mathbf{d}= C^{bl}_1 {\displaystyle
\frac{x_2 }{ \ep^\delta}}  \mathbf{e}^1 $.\vskip1pt

 \vskip1pt

Now,  after \cite{JaMi2}, we  expected that the approximation for the velocity reads
%to be  $ o(\ep )$ for the velocity and $O(1)$for the pressure:
\begin{gather}
\mathbf{v} (\ep ) = \mathbf{v}^0  -  \frac{F}{2} \ep^{\delta +1 -\gamma} \beta^{ bl} (\frac{x }{ \ep})  + \frac{F}{2} \ep^{\delta +1 -\gamma}\mathbf{d} = \notag \\
-{\ep^{2\delta -\gamma}  \frac{F }{ 2    }} \frac{x_2^+}{\ep^\delta} (\frac{x_2}{\ep^\delta} -
1) \mathbf{e}^1 -  \frac{F}{2} \ep^{\delta +1 -\gamma} \beta^{ bl} (\frac{x }{ \ep})  + \frac{F}{2} \ep^{\delta +1 -\gamma} C^{bl}_1 {
\frac{x_2^+ }{ \ep^\delta}}  \mathbf{e}^1
   . \label{4.66}
\end{gather}
\vskip2pt  Concerning the pressure,  there are additional complications due to the
stabilization of the boundary layer pressure to $C^{bl}_\omega $,
when $y_2 \to +\infty$. Consequently,
 $\displaystyle {\omega}^{ bl , \ep} - H(x_2) C^{bl}_\omega  \frac{\p   v^0_1 }{ \p    x_2 } |_\Sigma $ is small
in $\Omega_1^{\ep , \delta}$ and we should take into account the pressure stabilization effect.

At the flat
interface $\Sigma$, the normal component of the normal stress
reduces to the pressure field. Subtraction of the stabilization
pressure constant at infinity leads to the pressure jump on
$\Sigma$ and the pressure approximation is
\begin{equation}\label{BJ47}
p (\ep)  = - \frac{F}{2} \ep^\delta \bigl( {\omega}^{ bl} (\frac{x}{\ep}) -
C^{bl}_\omega \bigr)   .
\end{equation}
For the rigorous justification of the pressure approximation, leading to the pressure jump law, we refer to \cite{AMCAM2011} . Numerical experiments, justifying independently the { pressure jump} are in \cite{CGMCM:13}.
\vskip3pt
We now make the velocity calculations  rigorous. Let us define the errors in
velocity and in the pressure:
\begin{gather}
{\cal U}^\ep (x) =  \ve -  \mathbf{v} (\ep), \qquad %  \label{BJ4.66} \\
{\cal P}^\ep (x) = p^\ep - p (\ep) . \label{BJ4.67}
\end{gather}
 \begin{remark} Rigorous argument, showing that ${\cal U}^\ep$ is of order $O(\ep^{2-\gamma})$, allows  justifying  Saffman's
modification  of the Beavers and Joseph law (see \cite{BJ} and \cite{SAF}):
On the interface
$\Sigma$ we obtain
$$\frac{\p v_1 (\ep) }{ \p x_2 } |_\Sigma = - \ep^{\delta -\gamma} \frac{F}{2}(\frac{2x_2}{ \ep^\delta } -1)
|_\Sigma - \ep^{\delta -\gamma} \frac{F}{2} \frac{\p \beta^{bl}_1 }{ \p y_2} |_{\Sigma , y=x/\ep}  +\ep^{1 -\gamma} \frac{F}{2} C^{bl}_1  \; \hbox{ and } \; \frac{v_1 (\ep)}{ \ep } = -
\beta^{bl}_1 ( x_1 / \ep , 0)  \ep^{\delta -\gamma} \frac{F}{2}. $$ After averaging over $\Sigma$ with respect to $y_1$, we
obtain the %familiar form of the
Saffman version of the law by
Beavers and Joseph
  \begin{equation}\label{BJ}
u^{eff}_1 = -\ep C^{bl}_1 \frac{\p u^{eff}_{1} }{ \p x_2}  + O(\ep^{2-\gamma}) \quad
\hbox{ on } \quad \Sigma ,    \end{equation}
 where $u^{eff}_1$ is the average of { $v_1 ({\ep})$}
over the characteristic pore opening at the naturally permeable wall. The
higher order terms are neglected. Nevertheless, for { $\gamma $ close to $1$} the Beavers and Joseph slip law isn't satisfactory any more. \end{remark}

Next, the variational equation for $\{ {\cal U}^\ep , {\cal P}^\ep \} $
%reads  (\ref{4.37Couette})-(\ref{4.40Couette})
   reads
\begin{gather}
 \int_{\Oe} \ep^\gamma \nabla  {\cal U}^\ep :  \nabla \varphi \ dx  +\int_{\Oe}  \bigg( ( {\cal U}^\ep \nabla ) {\cal U}^\ep + ({\cal U}^\ep \nabla ) \mathbf{v} (\ep) + (\mathbf{v} (\ep)\nabla )  {\cal U}^\ep \bigg)\varphi \ dx \notag \\
  - \int_{\Oe} {\cal U}^\ep \mbox{ div } \varphi \ dx = - \int_{\Oe} (\mathbf{v} (\ep) \nabla ) \mathbf{v} (\ep) \varphi \ dx  -\int_{\Sigma} \ep \varphi_1    \frac{F}{2}  C^{bl}_1   \ dS , \; \forall \varphi \in  { W}^\ep .
\label{zsigma}
\end{gather}
Note that ${\cal U}^\ep$ is  divergence free  and the approximation satisfies the outer boundary conditions.
In analogy with Proposition 4, pages 1120-1121, from \cite{JM00} we have
\begin{theorem} \label{T4.17}Let us suppose the hypotheses ({\bf H1})-({\bf H2}) and  let ${\cal U}^\ep$   and
${\cal P}^\ep$ be defined by (\ref{BJ4.67}).
 Then, the following
estimates hold
\begin{gather}
 %\ep \Vert \nabla {\cal P}^\ep  \Vert_{H^{-1} (\Oe)} +
\ep \Vert \nabla {\cal U}^\ep  \Vert_{L^2 (\Oe )^4} + \Vert  {\cal U}^\ep  \Vert_{L^2 (\Oe_2 )^2} +
\ep^{1/2} \Vert  {\cal U}^\ep  \Vert_{L^2 (\Sigma  )^2} +\ep^{1-\delta} \Vert  {\cal U}^\ep  \Vert_{L^2 (\O_1^{\ep, \delta} )^2}
 \leq C\ep^{5/2-\gamma}
 \label{4.84} % \leq C\ep \label{4.85}
 \end{gather}
\end{theorem}
\begin{proof}
We test (\ref{zsigma}) by ${\cal U}^\ep$. Since  div ${\cal U}^\ep=0$, ${\cal P}^\ep$ is eliminated from the equality. Next, arguing as in the proof of Proposition \ref{P4.14}, we see that under assumptions ({\bf H1})-({\bf H2}) the viscous terms controls the inertia terms.
Therefore, it remains to estimate the forcing term and the interface term, coming from the counterflow.
We have
\begin{gather*}
    (\mathbf{v} (\ep) \nabla ) \mathbf{v} (\ep)   = - \frac{F}{2} \ep^{\delta +1 -\gamma} \Bigg( \bigg( -{\ep^{2\delta -\gamma}  \frac{F }{ 2    }} \frac{x_2^+}{\ep^\delta} (\frac{x_2}{\ep^\delta} -
1)  + \frac{F}{2} \ep^{\delta +1 -\gamma} C^{bl}_1 {
\frac{x_2^+ }{ \ep^\delta}}
- \frac{F}{2} \ep^{\delta +1 -\gamma} \beta^{bl}_1 (\frac{x}{\ep}) \bigg) \frac{\p \beta^{bl} (\frac{x}{\ep})}{\p x_1} \\
 - \frac{F}{2} \ep^{\delta +1 -\gamma} \beta^{bl}_2 (\frac{x}{\ep}) \frac{\p \beta^{bl} (\frac{x}{\ep})}{\p x_2} + \beta^{bl}_2 (\frac{x}{\ep}) \mathbf{e}^1 \frac{\p }{\p x_2} \big( -{\ep^{2\delta -\gamma}  \frac{F }{ 2    }} \frac{x_2^+}{\ep^\delta} (\frac{x_2}{\ep^\delta} -
1)  + \frac{F}{2} \ep^{\delta +1 -\gamma} C^{bl}_1 {
\frac{x_2^+ }{ \ep^\delta}}  \big)
\Bigg).
\end{gather*}
Since $\nabla_y \beta^{bl}$ decays exponentially in $y_2$ and the functions of $x_2$ behave as $x_2 \ep^{-\delta}$ for small $x_2$, { we obtain
\begin{gather}
    | \int_{\Oe}  \ep^{3\delta +1 -2\gamma} \frac{x_2^+}{\ep^\delta} (\frac{x_2}{\ep^\delta} -
1) \frac{\p \beta^{bl} (\frac{x}{\ep})}{\p x_1} {\cal U}^\ep \ dx | = | \int_{\Oe}  \ep^{3\delta +1 -2\gamma} \frac{x_2^+}{\ep^\delta} (\frac{x_2}{\ep^\delta} -
1) \frac{\p {\cal U}^\ep}{\p x_1}  \beta^{bl} (\frac{x}{\ep}) \ dx | \notag \\
 \leq C\ep^{2\delta -2\gamma + 5/2} || \nabla {\cal U}^\ep ||_{L^2 (\Oe)^4} \label{InertFirst}
\end{gather}
}
and the leading part in the first two terms of $(\mathbf{v} (\ep) \nabla ) \mathbf{v} (\ep)$ is
\begin{gather*}
   \frac{F^2}{4}   \ep^{2\delta +2 -2\gamma} (\beta^{bl} (\frac{x}{\ep})\nabla_x )\beta^{bl} (\frac{x}{\ep}).
\end{gather*}
Similarly, after integration by parts { in $\Omega_1^{\ep, \delta}$ and using that $\beta^{bl}$ is divergence free, we obtain the same order of $\ep$ estimate as (\ref{InertFirst}) for}
\begin{gather*}
    | \int_{\Oe} \beta^{bl}_2 (\frac{x}{\ep}) \mathbf{e}^1 \frac{\p }{\p x_2} \big( -{\ep^{2\delta -\gamma}  \frac{F }{ 2    }} \frac{x_2}{\ep^\delta} (\frac{x_2}{\ep^\delta} -
1)  + \frac{F}{2} \ep^{\delta +1 -\gamma} C^{bl}_1 {
\frac{x_2 }{ \ep^\delta}}  \big) \varphi \ dx |.
\end{gather*}
Consequently, it results in
\begin{gather}
    | \int_{\Oe} (\mathbf{v} (\ep) \nabla ) \mathbf{v} (\ep)  {\cal U}^\ep \ dx | \leq C\ep^{3\delta -2\gamma + 3/2} || \nabla {\cal U}^\ep ||_{L^2 (\Oe)^4} \label{Inert11}\\
    | \int_{\Sigma} \ep {\cal U}^\ep_1    \frac{F}{2}  C^{bl}_1   \ dS | \leq C \ep^{3/2}  || \nabla {\cal U}^\ep ||_{L^2 (\Oe)^4}.\label{Inert22}
\end{gather}
Applying Lemma \ref{L1} yields the estimate (\ref{4.84}).
\end{proof}
Still the { shear jump at the interface} dominates inertia due to the counterflow.
{ Correcting the shear jump term $\displaystyle -\int_{\Sigma} \ep \varphi_1    \frac{F}{2}  C^{bl}_1   \ dS$ is as above. The only difference is that instead of $\ep^\delta$ we have $\ep$ and $F/2$ is replaced by $-F C^{bl}_1 /2$.}
We eliminate it by  modifying slightly the velocity and pressure corrections:
\begin{corollary}
Let assumptions ({\bf H1})-({\bf H3}) hold,  and ${\cal U}^\ep$,
${\cal P}^\ep$ be defined by (\ref{BJ4.67}). Let
\begin{gather}
    {\cal U}^{1,\ep } =  {\cal U}^\ep  - \frac{F}{2} C^{bl}_1 \ep^{2 -\gamma} \beta^{ bl} (\frac{x }{ \ep})  + \frac{F}{2} \ep^{2 -\gamma} (C^{bl}_1)^2 {
\frac{x_2^+ }{ \ep^\delta}}  \mathbf{e}^1 , \label{corr1}
\end{gather}
 Then, the following
estimate holds
\begin{gather}
 %\ep \Vert \nabla {\cal P}^\ep  \Vert_{H^{-1} (\Oe)} +
\ep \Vert \nabla {\cal U}^{1,\ep }   \Vert_{L^2 (\Oe )^4} + \Vert  {\cal U}^{1,\ep }  \Vert_{L^2 (\Oe_2 )^2} +
\ep^{1/2} \Vert  {\cal U}^{1,\ep }  \Vert_{L^2 (\Sigma  )^2} +\ep^{1-\delta} \Vert  {\cal U}^{1,\ep }  \Vert_{L^2 (\O_1^{\ep, \delta} )^2}
 \leq C\ep^{5/2+ 3\delta -3\gamma}.
 \label{4.85}
 \end{gather}
\end{corollary}
{ The new shear stress jump term generated by  correction (\ref{corr1}) is given by $\displaystyle -\int_{\Sigma} \ep^{2-\delta} \varphi_1    \frac{F}{2}  (C^{bl}_1)^2   \ dS$. Then, the corresponding estimate (\ref{Inert22}) in the proof of Theorem \ref{T4.17} takes the form
\begin{equation}\label{Resshear}
    | \int_{\Sigma} \ep^{2-\delta} {\cal U}^\ep_1    \frac{F}{2}  (C^{bl}_1)^2   \ dS | \leq C \ep^{5/2 -\delta}  || \nabla {\cal U}^\ep ||_{L^2 (\Oe)^4}.
\end{equation}
Due to hypothesis ({\bf H3}), we have $5/2 -\delta > 3\delta -2\gamma +3/2$ and} the new error terms  are less important than the leading inertia terms.\vskip6pt
Finally, we correct the inertia term effects. { We note that it is multiplied by a small parameter $\displaystyle \ep^{\delta -2\gamma +2} $. We follow the idea from \cite{BMPAM:96} and expand the solutions to the nonlinear boundary layer problem (\ref{BJR5.2})-(\ref{BJR5.5}) in powers of that parameter. As already explained in the beginning of the section, the solutions of \ref{BJR5.2})-(\ref{BJR5.5})  take the form $ \displaystyle \{ \beta^0 +\frac{F}{2} \ep^{\delta -2\gamma +2} \beta^1 +\dots  , \pi^0 + \frac{F}{2} \ep^{\delta -2\gamma +2} \pi^1 +\dots \}$.
%{\color{blue}  but this sentence is semantically not correct and I do not understand it; what we controlled ? why we %in a past tense ? also contained  a small parameter multiplying it sound terrible; why not the inertia term was %multiplied by a small parameter ?}
Furthermore,  the $1$-periodicity of the geometry in $y_1$-direction allows to replace $\beta^0$ by $\beta^{bl}$.
It is similar with $\beta^1$. We recall that the  leading error term for ${\cal U}^{1, \ep}$ results from $(\beta^{bl} \nabla ) \beta^{bl} $.}
We introduce the boundary layer problem for $\beta^{1, bl}$:
\begin{gather} - \triangle _y \beta^{1, bl}   +  \nabla_y \pi^{1, bl} = (\beta^{bl}\nabla_y ) \beta^{bl} \qquad \hbox{ in
}  \quad  Z_{BL} ,\label{BLRRR5.2}\\
\div_y \beta^{1,bl} =0\qquad \hbox{ in }  \quad Z_{BL} , \label{BLRRR5.3} \\
\nabla_y \beta^{1, bl}  \in  L^2 ( Z_{BL} )^4  \quad \mbox{ and } \quad \beta^{1, bl}  \in  L^2_{loc} ( Z_{BL} )^2 , \label{BLRRRa} \\
\beta^{1, bl} =0 \quad \hbox{ on }
\cup_{k=1}^{\infty} ( \p Y_s -\{ 0,k \} ), \quad  \mbox{ and} \quad \{ \beta^{1, bl} , \pi^{1, bl} \} \,
\hbox{ is } 1 - \hbox{periodic in } y_1 .\label{BLRRR5.5}
\end{gather}
The forcing term decays exponentially. Following \cite{JaMi2}, we know
 that the system (\ref{BLRRR5.2})-(\ref{BLRRR5.5}) describes a
boundary layer, i.e. $\beta^{1, bl} $ and $ \omega^{1, bl} $
stabilize exponentially towards $C^{bl}_{11} \mathbf{e}^1$ and $C_{\pi 1}$ , when $\vert y_2\vert \to
\infty$.
Then, the correction reads
\begin{gather}
    {\cal U}^{2,\ep } =  {\cal U}^\ep  - \frac{F}{2} C^{bl}_1 \ep^{2 -\gamma} \beta^{ bl} (\frac{x }{ \ep})  + \frac{F}{2} \ep^{2 -\gamma} (C^{bl}_1)^2 {
\frac{x_2^+ }{ \ep^\delta}}  \mathbf{e}^1  + \\
  +(\frac{F}{2})^2 \ep^{2\delta +3 -3\gamma} \beta^{1, bl} (\frac{x }{ \ep})  { -} (\frac{F}{2})^2 \ep^{2\delta +3 -3\gamma} C^{bl}_{11} {
\frac{x_2^+ }{ \ep^\delta}}  \mathbf{e}^1 , \label{corr3}
\end{gather}
%and we have
 In complete analogy with Theorem \ref{T4.17} we prove Theorem \ref{T4.177}.
 \vskip0pt {  To obtain estimate (\ref{4.885}) from Theorem \ref{T4.177}, it is  enough to note that after (\ref{InertFirst}), the leading remaining inertia terms give a contribution bounded by}
 $$       C \ep^{2\delta +5/2 -2\gamma} || \nabla {\cal U}^{2,\ep } ||_{L^2 (\Oe)^4}        $$
  { Next, using hypothesis ({\bf H1}), we obtain that  $5/2 -\delta < 2\delta -2\gamma +5/2$. Furthermore,  the leading order term is the shear stress jump term
  $$  \int_{\Sigma} \ep^{2-\delta} \varphi_1    \frac{F}{2}  (C^{bl}_1)^2   \ dS.  $$
 It is estimated by (\ref{Resshear}), which yields (\ref{4.885}).  }

\end{document}